\documentclass[12pt]{article}
\usepackage{authblk}
\usepackage{pgf,tikz,pgfplots}
\pgfplotsset{compat=1.15}
\usepackage{mathrsfs}
\usetikzlibrary{arrows}
\usepackage{lineno}
\usepackage{amsmath,amssymb}%,psfig,mathrsfs}
\usepackage{amsthm} %theorem environment option
\usepackage{latexsym}
\usepackage{mathtools}
\usepackage{hyperref}
\newtheorem{theorem}{Theorem}[section]
\newtheorem{proposition}[theorem]{Proposition}
\newtheorem{lemma}[theorem]{Lemma}
\newtheorem{corollary}[theorem]{Corollary}

\newtheorem{definition}[theorem]{Definition}

\title{\bf A super-multiplicative inequality for the number of finite unlabeled arbitrary and $T_0$ topologies}

\author[1,2]{Ibtsam A. R. Alroily}
\author[1,*]{Brahim Chaourar}
\affil[1]{Department of Mathematics and Statistics, College of Science,\par Imam Mohammad Ibn Saud Islamic University (IMSIU), Riyadh, Saudi Arabia}
\affil[2]{Mathematics Department, College of Science,\par Jouf University, P.O. Box 2014, Sakaka, Saudi Arabia, earowily@ju.edu.sa}
\affil[*]{correspondence email: imchaourar@imamu.edu.sa}

\begin{document}
%\linenumbers
\renewcommand{\Affilfont}{\normalfont\small}
\date{}
\maketitle
\begin{abstract}
\noindent Let $n$ be a nonnegative integer, and $f(n)$ the number of unlabeled finite topologies on $n$ points. We prove that $f(n+m) \geq f(n) f(m)$ both for the labeled and unlabeled cases. Moreover, we prove a similar inequality for labeled and unlabeled $T_0$ topologies.
\end{abstract}

\section{Introduction}

Finite topological spaces raise interesting combinatorial questions, most notably the number $T(n)$ of distinct topologies on $n$ points. Exhaustive enumeration has established $T(n)$ for $n\leq 18$ \cite{BrinkmannMcKay:2005,oeis}, yet a general formula remains far for reaching.
\\ The enumeration can be further refined by counting $T(n,k)$, the number of topologies on $n$ points with $k$ open sets. This remains a long-standing open problem, though some known cases exist. Key contributions come from Ern\'e and Stege, who computed $T(n,k)$ for $n\leq 11$ and arbitrary $k$ in \cite{ErneStege:1990}, including related counts for $T_0$, connected topologies, and homeomorphism classes. Their results yielded all $T(n,k)$ for $k\leq 23$ \cite{ErneStege:1991}.
\\ Asymptotic behavior is complex. Finding a simple closed-form expression or a straightforward asymptotic formula for $T(n)$ has proven to be quite challenging. However, there are some known bounds and asymptotic estimates. For instance, it is known that $T(n)$ grows faster than $2^{\binom{n}{2}}$.
\\ However, very little is known about the growth behaviour of $T(n)$. Super-multiplicative inequalities involving $T(n+m)$, $T(n)$, and $T(m)$, for any nonnegative integers $n$ and $m$, are also important for this question. In another field of combinatorics, matroid theory, a famous conjecture is that: $f(n+m)\geq f(n) f(m)$, where $f(n)$ is the number of non-isomorphic matroids on $n$-element set \cite{Welsh:1969} (see \cite[p. 594]{Oxley:2011}). It resisted three decades \cite{CrapoSchmitt:2005, Lemos:2004}. If we prove a similar inequality for the number of unlabeled finite topologies on $n$ points, it would reveal some interesting structural properties about how topologies on disjoint sets combine to form topologies on their union. We can then get lower bounds on the growth of $f(n)$. For example, since $f(2) = 3$ \cite{oeis}, $f(n) = f(2 + (n-2)) \geq f(2) f(n-2) = 3 f(n-2)$, suggesting at least exponential growth with a base related to $\sqrt{3}$. Such inequalities could be a stepping stone towards a more precise understanding of the asymptotic growth rate. In addition of the analogy with matroid theory, this kind of inequality provides recursive approximation via maximum product. Super-multiplicativity enables recursive lower bounds:
\[
f(n) \geq \max_{1 \leq k \leq n-1} \left\{ f(k) \cdot f(n-k) \right\}.
\]
This formulation is useful for estimating \( f(n) \) when exact enumeration is difficult, and it reflects the idea that the most prolific decompositions dominate the growth.
\\ Let $\tau_i$, $i = 1, 2$, be two topologies defined on finite and disjoint sets $E_i$. The direct sum $\tau = \tau_1 \oplus \tau_2$ is the collection $\{ O_1 \cup O_2$ with $(O_1, O_2) \in \tau_1 \times \tau_2 \}$. A straightforward consequence is that the direct sum defines a topology on the disjoint union $E_1 \cup E_2$. For the labeled case, it is not difficult to see that, by using the direct sum, we achieve our super-multiplicative inequality. However, in the unlabeled case, we cannot prove it by using the direct sum. Indeed, if $n = m$, then we get $f(2 n) \geq \frac{1}{2} [f^2(n) + f(n)]$ by using this naive operation because it is commutative: if $\tau_i$, $i = 1, 2$, are two non-homeomorphic finite topologies, then $\tau_1 \oplus \tau_2 = \tau_2 \oplus \tau_1$. This is far from what we need to prove in this case: $f(2 n) \geq f^2(n)$. So, the direct sum fails here and a more elaborated (non-commutative) operation should be used for this purpose. It is what we call the $w$-sum. This operation is introduced in the coming section.
\\ Finite topologies have found concrete applications across a range of scientific and technological domains. In chemistry, they support the analysis of molecular graphs and topological indices used to predict chemical properties \cite{MerrifieldSimmons:1989}. In image analysis, finite topological spaces provide a rigorous framework for modeling pixel connectivity and digital surfaces \cite{Kovalevsky:1992}. In automata theory, they aid in the classification and minimization of state-transition systems, with implications for control and automation \cite{SinghMahato:2023}. Beyond these, finite topologies play a role in robotics, where they inform configuration space analysis and motion planning \cite{Latombe:1991}; in geographic information systems (GIS), where they model spatial relationships and adjacency structures \cite{Egenhofer:1991}; and in computer science, particularly in topological data analysis (TDA), where they contribute to the study of shape and connectivity in high-dimensional data \cite{Carlsson:2009}. These applications intersect with industrial sectors such as petrochemicals, electronics, automation, geospatial technology, and data science.
\\ The rest of the paper is organized as follows. In Section 2, we introduce the $w$-sum and prove some of its properties. In Section 3, we prove the main result. The last section is devoted to conclusions and further directions.

%%%%%%%%%%%%%%%%%%%%%%%%%%%%%%%%%%%%%%%%%%%%%%%%%%%%%%%%%%%%%%%%%%%%%%%%%%%%%%%%%%%%%%%%%%%%%%%%%%%%%%%%%%%%%%%%%%%%%%%%%%%%%%%%%%%%%%%%%%%%%%%%%%%%%%%%%%%%%%%%%%%%%%%%%%
\section{The $w$-sum operation}

\begin{definition}\label{Topo}
  Let $n$ be a nonnegative integer, $E$ a finite set of cardinality $n$, and $\tau$ a collection of subsets of $E$. We say that $\tau$ is a topology on $E$, or on $n$ points, if:
  \\ (w1) $\O , E \in \tau$;
  \\ (w2) $\tau$ is closed under union, that is, if $A, B \in \tau$, then $A \cup B \in \tau$;
  \\ (w3) $\tau$ is closed under intersection, that is, if $A, B \in \tau$, then $A \cap B \in \tau$.
  \\ Moreover, the size of $\tau$ is $n$, $E$ is called its ground set, and its members are called open sets.
\end{definition}

Let us recall what we mean by a topologies homeomorphism.

\begin{definition}\label{TopoIso}
  Let $\tau_i$, $i = 1, 2$, be two topologies defined on finite sets $E_i$, respectively, and $\varphi$ a mapping from $E_1$ to $E_2$. We say that $\varphi$ is an homeomorphism from $\tau_1$ to $\tau_2$ if:
  \\ (1) $\varphi$ is bijective;
  \\ (2) $Y = \varphi (X) \in \tau_2$ if and only if $X \in \tau_1$.
  \\ In this case, we say that $X$ and $Y$ are homeomorphic, and $\tau_1$ and $\tau_2$ are homeomorphic.
  \\ In general, two topologies are homeomorphic if such homeomorphism exists.
  \\ The class of unlabeled topologies is the class of non-homeomorphic topologies.
\end{definition}

A direct consequence of the above definition is

\begin{proposition}\label{cardHomeo}
  Two homeomorphic open sets have the same cardinality. Moreover, homeomorphisms keep inclusion (order), intersection, union, and complement.
\end{proposition}

We denote by $T(n)$ and $f(n)$ the number of labeled and unlabeled topologies on $n$ points, respectively. While the corresponding class for unlabeled topologies is denoted by $\mathcal T(n)$.

\begin{definition}\label{shiftTopo}
  Let $\tau$ be a finite topology defined on $E$, $X \in \tau$, and $Y$ a finite set that is disjoint from $E$.
  \\ (1) The intersection-topology $\tau \cap X$ is the collection $\{ O \cap X$ for all $O \in \tau \}$.
  \\ (2) The shift topology $\tau - X$ is the collection $\{ O \backslash X$ for all $O \in \tau \}$.
  \\ (3) The inverse shift topology $\tau + Y$ is the collection $\{ O \cup Y$ for all $O \in \tau \}$.
\end{definition}

We recall the notion of co-topology.

\begin{definition}
  Let $\tau$ be a finite topology defined on $E$. Its co-topology denoted $\tau^C$ is the collection $\{ E \backslash X$ for all $X \in \tau \}$.
\end{definition}

It is evident that a co-topology is again a topology defined on the same ground set. From the above definitions, we deduce the following properties.

\begin{lemma}\label{subWtopo}
  Let $\tau$ be a finite topology defined on $E$, $X \in \tau$, and $Y$ a finite set that is disjoint from $E$. Then
  \\ (i) $\tau \cap X$ is a finite topology on $X$.
  \\ (ii) $\tau - X = [\tau^C \cap (E \backslash X)]^C$.
  \\ (iii) $\tau - X$ is a finite topology on $E \backslash X$.
  \\ (iv) $\tau + Y$ is closed under union and intersection.
\end{lemma}
\begin{proof}\
  \\ (i) Let $A, B \in \tau \cap X$. It follows that $A, B \subseteq X$. This yields $A \cup B, A \cap B \subseteq X$. Since $\O, X \in \tau (X) \subseteq \tau$, $\tau \cap X$ is a topology on $X$. This is what we request.
  \\ (ii) Let $Y \in \tau - X$, that is, $Y = O \backslash X$, for some $O \in \tau$. Thus, $(E \backslash X) \backslash Y = (E \backslash X) \cap (E \backslash O) \in (E \backslash X) \cap \tau^C = \tau^C \cap (E \backslash X)$. In other words, $Y \in [\tau^C \cap (E \backslash X)]^C$, and vice versa.
  \\ (iii) Since a co-topology is also a topology on the same ground set and according to (i)-(ii), $\tau - X$ is a topology on $E \backslash X$.
  \\ (iv) Similarly as for (i), we can prove that $\tau + Y$ is closed under union and intersection. Its smallest member is $Y$, while its larger one is $E \cup Y$.
\end{proof}

\begin{definition}\label{T0Topo}
  We say that a topology $\tau$ defined on $E$ is a $T_0$ one, if, in addition, it satisfies the following condition:
\\ (t0) For any two distinct points $x, y \in E$, there exists an open set $A \in \tau$, such that $|\{ x, y \} \cap A| = 1$.
\end{definition}

The numbers of unlabeled and labeled $T_0$-topologies on $n$ points are denoted by: $f_{0} (n)$, and $T_{0} (n)$, respectively.
\\ We have analogous results as for Lemma \ref{subWtopo}.

\begin{lemma}\label{subT0topo}
  Let $\tau$ be a $T_0$ topology defined on $E$, $X \in \tau$, and $Y$ a finite set that is disjoint from $E$. Then
  \\ (i) $\tau \cap X$ is a $T_0$ topology.
  \\ (ii) $\tau^C$ is a $T_0$ topology.
  \\ (iii) $\tau - X$ is a $T_0$ topology.
\end{lemma}
\begin{proof}\
  \\ (i) Let $x, y \in X \subseteq E$ be two distinct points. It follows that there exists an open set $A \in \tau$, such that $|\{ x, y \} \cap A| = 1$. This yields $A \cap X \in \tau \cap X$. Furthermore, $|\{ x, y \} \cap (A \cap X)| = |(\{ x, y \} \cap X) \cap A| = |\{ x, y \} \cap A| = 1$. This is what we request.
  \\ (ii) Let $x, y \in E$ be two distinct points. It follows that there exits an open set $A \in \tau$, such that $|\{ x, y \} \cap A| = 1$. This means that $|\{ x, y \} \cap (E \backslash A)| = 1$ because $|\{ x, y \} \cap E| = 2$. Since $E \backslash A \in \tau^C$, we are done.
  \\ (iii) Combining (i)-(ii) of the current lemma with (ii) of Lemma \ref{subWtopo} imply the result.
\end{proof}

Now, we define the $w$-sum of two finite topologies.

\begin{definition}\label{wSum}
  Let $E_i$, $i = 1, 2$, be two disjoint finite sets, and $\tau_i$ two topologies defined on $E_i$, with $|E_i| = n_i$, respectively. The $w$-sum of $\tau_1$ and $\tau_2$ is

  $$ \tau_1 \oplus_w \tau_2 = \tau_1 \cup (\tau_2 + E_1) .$$

\end{definition}

It is clear that the $w$-sum is not commutative in general. Furthermore,

\begin{proposition}\label{taui}\
  \\ If $\tau = \tau_1 \oplus_w \tau_2$ is a topology, then $\tau_1 = \tau \cap E_1$, and $\tau_2 = \tau - E_1$.
\end{proposition}
\begin{proof}\
  \\ Let $A \in \tau_1 \subseteq \tau$. This means that $A \subseteq E_1$, i.e., $A \in \tau \cap E_1$, and vice versa. For $B \in \tau_2$, we have, $B \cup E_1 \in \tau_2 + E_1 \subseteq \tau$, i.e., $B = (B \cup E_1) \backslash E_1 \in \tau - E_1$, and vice versa.
\end{proof}

\begin{lemma}\label{sumw}
  Let $\tau_i$, $i = 1, 2$, be two finite topologies defined on two disjoint sets, and $\tau = \tau_1 \oplus_w \tau_2$ a topology. Then the following assertions are equivalent.
  \\ (i) $\tau$ is a $T_0$ topology.
  \\ (ii) $\tau_1$ and $\tau_2$ are $T_0$ topologies.
\end{lemma}
\begin{proof}\
\\ {\bf (i) $\Rightarrow$ (ii):} Proposition \ref{taui} implies that $\tau_1 = \tau \cap E_1$ and $\tau_2 = \tau - E_1$. They are two $T_0$ topologies according to (i) and (iii) of Lemma \ref{subT0topo} when $\tau$ is a $T_0$ one.
\\ {\bf (ii) $\Rightarrow$  (i):} Since $\tau = \tau_1 \cup [E_1 + \tau_2]$, its largest open set is $E = E_1 \cup E_2$ and its smallest one is the empty set. Let $x, y \in E$ be two distinct points. If both belong to $E_1$ ($E_2$, respectively), then there exists an appropriate open set $X \in \tau_1$ or ($Y \in \tau_2$, i.e., $Y \cup E_1 \in \tau_2 + E_1$, respectively) that distinguish them. Now if $x \in E_1$ and $y \in E_2$, then $y \notin E_1$, and $E_1 \in \tau_1 \subseteq \tau$ is the right open set since $E_1 \cap E_2 = \O$.
\end{proof}

The main result of this section is

\begin{theorem}\label{Wsum}\
  \\ (1) $\tau = \tau_1 \oplus_w \tau_2$ is a topology defined on $E_1 \cup E_2$, that is, a topology on $|E_1| + |E_2|$ points.
  \\ (2) If $\tau_1$ and $\tau_2$ are $T_0$ topologies, then so is $\tau$.
\end{theorem}
\begin{proof}\
  \\ (1) Let $A, B \in \tau$. Without loss of generality, we can suppose that $A \in \tau_1$, and $B \in \tau_2 + E_1$ because the latter is closed under union and intersection according to (iv) of Lemma \ref{subWtopo}. In this case,
  $$ B = O_2 \cup E_1 \supseteq A \eqno (1), $$
  for some $O_2 \in \tau_2$. Hence, $A \cup B = B \in \tau$, and $A \cap B = A \in \tau$. In other words, $\tau$ is a topology on $E_1 \cup E_2 \in E_1 + \tau_2$ since $\O \in \tau_1 \subseteq \tau$. So, its size is $|E_1 \cup E_2| = |E_1| + |E_2|$, and we are done.
  \\ (2) is a part of Lemma \ref{sumw}.
\end{proof}

%%%%%%%%%%%%%%%%%%%%%%%%%%%%%%%%%%%%%%%%%%%%%%%%%%%%%%%%%%%%%%%%%%%%%%%%%
\section{Main result}

First, we prove that intersection-topologies and shift topologies keep homeomorphism.

\begin{lemma}\label{subtopoIsomorphism}
  Let $\tau_i$ be two homeomorphic topologies, and $X_i \in \tau_i$, $i = 1, 2$, two homeomorphic open sets. Then
  \\ (i) $\tau_1 \cap X_1$ and $\tau_2 \cap X_2$ are homeomorphic.
  \\ (ii) $\tau_1^C$ and $\tau_2^C$ are homeomorphic.
  \\ (iii) $\tau_1 - X_1$ and $\tau_2 - X_2$ are homeomorphic.
\end{lemma}
\begin{proof}\
  \\ Let $\varphi$ be a homeomorphism from $\tau_1$ to $\tau_2$. This means that $\varphi$ is bijective from $E_1$ to $E_2$ and $\varphi (Y_1) \in \tau_2$ if and only if $Y_1 \in \tau_1$.
  \\ (i) Now let $\varphi_{X_1}$ the restriction of $\varphi$ to $X_1 \subseteq E_1$. Since $X_1$ and $X_2$ are homeomorphic, i.e., they have the same cardinality, $\varphi_{X_1}$ is bijective from $X_1$ to $X_2$. Inclusions are kept by $\varphi$ yield $\tau_1 \cap X_1$ and $\tau_2 \cap X_2$ are homeomorphic.
  \\ (ii) For the co-topologies, consider the same bijective mapping $\varphi$ from $E_1$ to $E_2$. Since an homeomorphism keeps intersection, union, and complement, $\varphi (E_1 \backslash O_1) = \varphi (E_1 \cap (E_1 \backslash O_1)) =  \varphi (E_1) \cap \varphi (E_1 \backslash O_1) = E_2 \cap [\varphi (O_1)]^C = E_2 \cap (E_2 \backslash O_2) = E_2 \backslash O_2 \in \tau_2^C$ if and only if $E_1 \backslash O_1 \in \tau_1^C$.
  \\ (iii) (i)-(ii) of the current lemma and (ii) of Lemma \ref{subWtopo} permits us to conclude.
\end{proof}

We denote by $A \times B$ the classical cartesian product for two sets $A$ and $B$. We introduce the following mapping $\varphi$ from $\mathcal T(n) \times \mathcal T(m)$ to $\mathcal T(n + m)$ as follows. For any $(\tau_1, \tau_2) \in \mathcal T(n) \times \mathcal T(m)$, $\varphi (\tau_1, \tau_2) = \tau_1 \oplus_w \tau_2$

\begin{theorem}\label{WsumInjective}
  $\varphi$ is an injective mapping.
\end{theorem}
\begin{proof}\
  \\ First, it is clear that $\varphi$ is well defined according to Theorem \ref{Wsum}. To prove its injectivity, consider $\tau_{ij} \in \mathcal T(i)$, $i = n, m$, and $\tau_j = \tau_{nj} \oplus_w \tau_{mj}$, $j = 1, 2$. Suppose that $\tau_1$ and $\tau_2$ are homeomorphic, and the ground sets of $\tau_{i1}$ are $E_{i1}$, $i = n, m$, respectively, while those of $\tau_{j2}$ are $F_{j1}$, $j = n, m$, respectively. Theorem \ref{Wsum} and (i)-(iii) of Lemma \ref{subtopoIsomorphism} imply that:
  \\ (i) $\tau_{n1} = \tau_1 \cap E_{n1}$ and $\tau_{n2} = \tau_2 \cap F_{n1}$ are homeomorphic.
  \\ (ii) $\tau_{m1} = \tau_1 - E_{n1}$ and $\tau_{m2} = \tau_2 - F_{n1}$ are homeomorphic.
  \\ Thus, $\varphi$ is injective.
\end{proof}

\begin{corollary}\label{Inequality}\
  \\ (i) $f(n + m) \geq f(n) f(m)$.
  \\ (ii) $f(n + m) \geq \max \{ f(i) f(m + n - i)$ for all $1 \leq i \leq n + m - 1 \}$.
\end{corollary}

\begin{corollary}\label{labeledInequality}\
  \\ (i) $T(n + m) \geq T(n) T(m)$.
  \\ (ii) $T(n + m) \geq \max \{ T(i) T(m + n - i)$ for all $1 \leq i \leq n + m - 1 \}$.
\end{corollary}

Now, we give the corresponding inequalities for $T_0$ topologies.

\begin{corollary}\label{Inequality}\
  \\ (i) $f_0(n + m) \geq f_0(n) f_0(m)$.
  \\ (ii) $f_0(n + m) \geq \max \{ f_0(i) f_0(m + n - i)$ for all $1 \leq i \leq n + m - 1 \}$.
\end{corollary}

\begin{corollary}\label{labeledInequality}\
  \\ (i) $T_0(n + m) \geq T_0(n) T_0(m)$.
  \\ (ii) $T_0(n + m) \geq \max \{ T_0(i) T_0(m + n - i)$ for all $1 \leq i \leq n + m - 1 \}$.
\end{corollary}

%%%%%%%%%%%%%%%%%%%%%%%%%%%%%%%%%%%%%%%%%%%%%%%%%%%%%%%%%%%%%%%%%%%%%%%%%%%%%%%%%%%%%%%%%%%
\section{Conclusion}

We have proved a super-multiplicative inequality for the number of finite topologies and $T_0$ ones in the labeled and unlabeled cases. While our lower bounds may not be sharp due to the use of simple formulas, our goal was to prove a similar non-trivial inequality as for matroids. To enhance our understanding, we can focus on exploring a wider variety of configurations for topologies on $n + m$ points, starting from two distinct topologies on $n$ and $m$ points. On the other hand, values of $T(n)$ and $f(n)$ are known for $n \leq 18$ and $n \leq 16$, respectively. So, our inequality gives lower bounds for unknown values of $T(n)$ and $f(n)$ when $19 \leq n \leq 35$ and $17 \leq n \leq 31$, respectively (and also for $T_0(n)$ and $f_0(n)$). Moreover, since $T_0$ topologies are in one-to-one correspondence with posets, we have a similar inequality for the number of non-isomorphic posets. Further investigations can be refining the obtained inequality by introducing a summation of products instead of a product only.

%%%%%%%%%%%%%%%%%%%%%%%%%%%%%%%%%%%%%%%%%%%%%%%%%%%%%%%%%%%%%%%%%%%%%%%%%%%%%%%%%%%%%%
\iffalse

\vspace{11pt}
{\Large\bf Acknowledgements}

This work was supported and funded by the Deanship of Scientific Research at Imam Mohammad Ibn Saud Islamic University (IMSIU) (grant number IMSIU-DDRSP2503).

The author is grateful to anonymous referees for their valuable remarks and corrections in a previous version of this paper.

\fi
%%%%%%%%%%%%%%%%%%%%%%%%%%%%%%%%%%%%%%%%%%%%%%%%%%%%%%%%%%%%%%%%%%%%%%%%%%%%%%%%%%%%%%%%%%%%%%%%%%%%%%%%%%%%%%%%%%%%%%%%%%%%%%%%%%%%%%%%%%%%%%
%\bibliographystyle{plain}
%\bibliography{bib}

\begin{thebibliography}{1}

%\bibitem{Benoumhani:2006}
%M. Benoumhani (2006), {\em The Number of Topologies on a Finite Set}, Journal of Integer Sequences 9, Article 06.2.6.

\bibitem{BrinkmannMcKay:2005}
G. Brinkmann and B. D. McKay (2005), {\em Counting unlabeled topologies and transitive relations}, Journal of Integer Sequences 8, Article 05.2.1.

\bibitem{Carlsson:2009}
G. Carlsson (2009), {\em  Topology and data}, Bulletin of the American Mathematical Society 46 (2), 255-308. DOI: 10.1090/S0273-0979-09-01249-X.

\bibitem{CrapoSchmitt:2005}
H. Crapo and W. Schmitt (2005), {\em The free product of matroids}, European Journal of Combinatorics 26, 1060-1065.

\bibitem{Egenhofer:1991}
M. J. Egenhofer and R. D. Franzosa (1991), {\em Point-set topological spatial relations}, International Journal of Geographical Information Systems 5 (2), 161-174. DOI: 10.1080/02693799108927841.

\bibitem{ErneStege:1990}
M. Erne\'{e} and K. Stege (1990), {\em Counting finite posets and topologies}, Technical Report 236, Institute of Mathematics, University of Hannover, 1990.

\bibitem{ErneStege:1991}
M. Ern\'e and K. Stege (1991), {\em Counting finite posets and topologies}, Order 8, 247-265.

\bibitem{Kovalevsky:1992}
V.A. Kovalevsky (1992), {\em Finite Topology and Image Analysis}, Advances in Electronics and Electron Physics 84, 197-259.

\bibitem{Latombe:1991}
J.-C. Latombe (1991), {\em Robot Motion Planning}, Springer, Boston, MA.

\bibitem{Lemos:2004}
M. Lemos (2004), {\em On the number of non-isomorphic matroids}, Advances in Applied Mathematics 33, 733-746.

\bibitem{MerrifieldSimmons:1989}
R.E. Merrifield and H.E. Simmons, Topological Methods in Chemistry, Wiley, New York, 1989.

\bibitem{Oxley:2011}
J. G. Oxley (2011), {\em Matroid Theory}, Oxford University Press, Oxford.

\bibitem{SinghMahato:2023}
S. Singh and S. P. Tiwari, and S. Mahato (2023), {\em On L-fuzzy automata, coalgebras and dialgebras: Associated categories and L-fuzzy topologies}, Fuzzy Sets and Systems 469, 1-27. \href{https://10.1016/j.fss.2023.01.005}{link}.

\bibitem{oeis}
N. J. A. Sloane, {\em The On-Line Encyclopedia of Integer Sequences}, \href{https://oeis.org/A000798/list}{https://oeis.org/A000798/list} (labeled), \href{https://oeis.org/A001930/list}{https://oeis.org/A001930/list} (unlabeled).

\bibitem{Stanley:1997}
R. P. Stanley (1997), {\em Enumerative Combinatorics}, Vol. 1, Cambridge University Press, Cambridge, United Kingdom.

\bibitem{Welsh:1969}
D. J. A. Welsh (1969), {\em A bound for the number of matroids}, Journal of Combinatorial Theory (Series B) 6 (3), 313-316.

\end{thebibliography}

\end{document}